\documentclass{amsart}
\usepackage{mathrsfs,amssymb,mathrsfs, multirow,setspace}
\usepackage{amsmath}
\usepackage{listings}
\usepackage{pdfpages}
\usepackage{caption}
\usepackage[a4paper, total={7in, 9in}]{geometry}
\usepackage{enumerate}
\theoremstyle{plain}
\usepackage{graphicx}
\usepackage{mathtools}
\usepackage[utf8]{inputenc}

\newtheorem{theorem}{Theorem}[section]
\newtheorem{lemma}[theorem]{Lemma}
\newtheorem{proposition}[theorem]{Proposition}
\newtheorem{corollary}[theorem]{Corollary}

\theoremstyle{definition}

\newtheorem{example}[theorem]{Example}

\theoremstyle{remark}

\input xy
\xyoption{all}

\begin{document}
	\title[Wiener index of the Cozero-divisor graph of a finite commutative ring]{Wiener index of the Cozero-divisor graph of a finite commutative ring}
	\author[Barkha Baloda, Praveen Mathil, Jitender Kumar, Aryan Barapatre]{Barkha Baloda, Praveen Mathil, $\text{Jitender Kumar}^{^*}$, Aryan Barapatre}
	\address{Department of Mathematics, Birla Institute of Technology and Science Pilani, Pilani, India}
	\email{barkha0026@gmail.com, maithilpraveen@gmail.com,  jitenderarora09@gmail.com, aryanbar@gmail.com}

\begin{abstract}
Let $R$ be a ring with unity. The cozero-divisor graph of a ring $R$, denoted by
$\Gamma'(R)$, is an undirected simple graph whose vertices are the set of all
non-zero and non-unit elements of $R$, and two distinct vertices $x$ and $y$  are
adjacent if and only if $x \notin Ry$ and $y \notin Rx$. In this article, we extend some of the results of \cite{a.mathil2022cozero} to an arbitrary ring. In this connection, we derive a closed-form formula of the Wiener index of the cozero-divisor graph of a finite commutative ring $R$. As applications, we compute the Wiener index of $\Gamma'(R)$, when either $R$ is the product of ring of integers modulo $n$ or a reduced ring. At the final part of this paper, we provide a SageMath code to compute the Wiener index of the cozero-divisor graph of these class of rings including the ring $\mathbb{Z}_{n}$ of integers modulo $n$.
\end{abstract}


\keywords{Cozero-divisor graph, Wiener index, reduced ring, ring of integer modulo $n$  \\ *  Corresponding author}

\maketitle

\section{Introduction and Preliminaries}

The Wiener index is one of the most frequently used topological indices in chemistry as a molecular shape descriptor. This was first used by H. Wiener in 1947 and then the formal definition of the Wiener index was introduced by Hosaya \cite{hosoya1971topological}. The \emph{Wiener index} of a graph is defined as the sum of the lengths of the shortest paths between all pairs of vertices in a graph.  Other than the chemistry, the Wiener index was used to find various applications in quantitative structure-property relationships (see \cite{karelson2000molecular}). The Wiener index was also employed in crystallography, communication theory, facility location, cryptography etc. (see \cite{bonchev2002wiener,gutman1993some, nikolic1995wiener}). An application of the Wiener index has been established in water pipeline network which is essential for water supply management (see \cite{dinar2022wiener}).  Other utilization of the Wiener index can be found in \cite{dobrynin2001wiener, janezic2015graph,wiener1947structural,xu2014survey} and reference therein.

The idea of associating a graph with a ring structure was first emerged in \cite{beck1988coloring}. Then various graphs associated with rings have been studied extensively in the literature, viz. inclusion ideal graph \cite{akbari2015inclusion}, total graph \cite{anderson2008total}, zero-divisor graph \cite{anderson1999zero}, annihilating-ideal graph \cite{behboodi2011annihilating, jalali2022strong}, co-maximal graph \cite{maimani2008comaximal}, etc. Afkhami \emph{et al.} \cite{afkhami2011cozero} introduced the cozero-divisor graph of a commutative ring, in which they have studied the basic graph-theoretic properties including completeness, girth, clique number, etc. of the cozero-divisor graph. They also studied the relations between the zero-divisor graph and the cozero-divisor graph. The cozero-divisor graph of a ring $R$ with unity, denoted by $\Gamma'(R)$, is an undirected simple graph whose vertex set is  the set of all non-zero and non-unit elements of $R$ and two distinct vertices $x$ and $y$  are adjacent if and only if $x \notin Ry$ and $y \notin Rx$. The complement of the cozero-divisor graph and the characterization of the commutative rings with forest, star, or unicyclic cozero-divisor graphs have been investigated in \cite{afkhami2012cozero}. Akbari \emph{et al.} \cite{akbari2014some} studied the cozero-divisor graph associated to the polynomial ring and the ring of power series. Some of the work associated with the cozero-divisor graphs of rings can be found in \cite{afkhami2012planar, afkhami2013cozero, akbari2014commutative, bakhtyiari2020coloring, mallika2017rings, a.mathil2022cozero, nikandish2021metric}.

Over the recent years, the Wiener index of certain graphs associated with rings have been studied by various authors.  The Wiener index of the zero divisor graph of the ring $\mathbb{Z}_n$ of integers modulo $n$ has been studied in \cite{a.asir2022wiener}. Recently, Selvakumar \emph{et al.} \cite{a.selva2022generalwiener} calculated the Wiener index of the zero divisor graph for a finite commutative ring with unity. The Wiener index of the cozero-divisor graph of the ring $\mathbb{Z}_n$ has been obtained in \cite{a.mathil2022cozero}. In order to extend the results of \cite{a.mathil2022cozero} to an arbitrary ring, we study the Wiener index of the cozero-divisor graph of a finite commutative ring with unity. First, we provide the necessary results and notations used throughout the paper. The remaining paper is arranged as follows: In Section 2, a closed-form formula of the Wiener index of the cozero-divisor graph of a finite commutative ring with unity is presented. In Section 3, we obtain the Wiener index of the cozero-divisor graph of the product of a ring of integers modulo $n$. In Section 4, we calculate the Wiener index of the cozero-divisor graph of a finite commutative reduced ring. In Section 5, we derive a SageMath code to compute the Wiener index of the cozero-divisor graph of various classes of rings. 

Now we recall necessary definitions, results and notations of graph theory from \cite{westgraph}. A graph $\Gamma$ is a pair  $ \Gamma = (V, E)$, where $V = V(\Gamma)$ and $E = E(\Gamma)$ are the set of vertices and edges of $\Gamma$, respectively. Let $\Gamma$ be a graph. Two distinct vertices $x, y \in \Gamma$ are $\mathit{adjacent}$, denoted by $x \sim y$, if there is an edge between $x$ and $y$. Otherwise, we denote it by $x \nsim y$.
 A \emph{subgraph} $\Gamma'$ of a graph $\Gamma$ is a graph such that $V(\Gamma') \subseteq V(\Gamma)$ and $E(\Gamma') \subseteq E(\Gamma)$. If $U \subseteq V(\Gamma)$ then the subgraph of $\Gamma$ induced by $U$, denoted by $\Gamma(U)$, is the graph with vertex set $U$ and two vertices of $\Gamma(U)$ are adjacent if and only if they are adjacent in $\Gamma$. The \emph{complement} $\overline{\Gamma}$ of $\Gamma$ is a graph with same vertex set as $\Gamma$ and distinct vertices $x, y$ are adjacent in $\overline{\Gamma}$ if they are not adjacent in $\Gamma$. A graph $\Gamma$ is said to be $complete$ if every two distinct vertices are adjacent. The complete graph on $n$ vertices is denoted by $K_n$. A path in a graph is a sequence of distinct vertices with the property that each vertex in the sequence is adjacent to the next vertex of it. The graph $\Gamma$ is said to be \emph{connected} if there is path between every pair of vertex. The distance between any two vertices $x$ and $y$ of $\Gamma$, denoted by $d(x,y)$ (or $d_{\Gamma}(x,y)$), is the number of edges in a shortest path between $x$ and $y$. The Wiener index is defined as the sum of all distances between every pair of vertices in the graph that is the Wiener index of a graph $\Gamma$ is given by
\[ W(\Gamma) = \dfrac{1}{2}\sum_{u \in V(\Gamma)}\sum_{v \in V(\Gamma)} d(u,v) \]
 
Let $\Gamma_1, \Gamma_2, \ldots, \Gamma_k$ be $k$ pairwise disjoint graphs. Then the \emph{generalised join graph} $\Gamma[\Gamma_1, \Gamma_2, \ldots, \Gamma_k]$ of $\Gamma_1, \Gamma_2, \ldots, \Gamma_k$ is the graph formed by replacing each vertex $u_i$ of $\Gamma$ by $\Gamma_i$ and then joining each vertex of $\Gamma_i$ to every vertex of $\Gamma_j$ whenever $u_i \sim u_j$ in $\Gamma$ (cf. \cite{schwenk1974computing}). The set of zero-divisors and the set of units of the ring $R$ is denoted by $Z(R)$ and $U(R)$, respectively. The  set of all nonzero elements of $R$ is denoted by $R^{*}$. For $x \in R$, we write $(x)$ as the principal ideal generated by $x$. For a positive integer $k$, we write $[k] = \{ 1,2, \ldots, k \} $.

\section{Formulae for The Wiener index of the cozero-divisor graph of a finite commutative ring}

The purpose of this section is to provide the closed-form formula of the Wiener index of the cozero-divisor graph of a finite commutative ring. Let $R$ be a finite commutative ring with unity. Define a relation $\equiv$ on $V(\Gamma'(R))$ such that  $x \equiv y$ if and only if $(x) = (y)$. Note that the relation $\equiv$ is an equivalence relation. Let $x_1, x_2, \ldots, x_k$ be the representatives of the equivalence classes of $X_1, X_2, \ldots, X_k$ respectively,  under the relation $\equiv$. We begin with the following lemma.


\begin{lemma}\label{adjacencyofclasses}
A vertex of $X_i$ is adjacent to a vertex of $X_j$ if and only if $(x_i) \nsubseteq (x_j)$ and $(x_j) \nsubseteq (x_i)$.
\end{lemma}
\begin{proof}
 Suppose $a \in X_i$ and $b \in X_j$. Then $(a) = (x_i)$ and $(b) = (x_j)$ in $R$. If $a \sim b$ in $\Gamma'(R)$, then $(a) \not\subset (b)$ and $(b) \not\subset (a)$. It follows that $(x_i) \not\subset (x_j)$ and $(x_j) \not\subset (x_i)$. The converse holds by the definition of $\Gamma'(R)$. 
\end{proof}

\begin{corollary}\label{corllary_wiener_adjacency}
\begin{itemize}
\item[(i)] For $i \in \{1, 2,\ldots, k\}$, the induced subgraph $\Gamma'(X_i)$ of $\Gamma'(R)$ is isomorphic to $\overline{K}_{|X_i|}$.

\item[(ii)] For distinct $i,j \in \{1, 2, \ldots, k\}$, a vertex of $X_i$ is adjacent to either all or none of the vertices of $X_j$. 
\end{itemize}
\end{corollary}

Define a subgraph $\Upsilon'(R)$ (or $\Upsilon'$) induced by the set $ \{x_1, x_2, \ldots, x_k \}$ of representatives of the respective equivalence classes $X_1, X_2, \ldots, X_k$ under the relation $\equiv$.  

\begin{lemma}\label{connectednessof_Upsilon}
The graph $\Upsilon'(R)$ is connected if and only if the cozero-divisor graph $\Gamma'(R)$ is connected. Moreover, for $a,b \in V(\Gamma'(R))$, we have
 \[ d_{\Gamma'(R)}(a,b) =
  \begin{cases}
 2  & \;\; \; \textnormal{if } a, b \in X_i, \\
d_{\Upsilon'(R)}(x_i, x_j)  & \;\;\; \textnormal{if } a \in X_i, b \in X_j~~ \textnormal{and} ~~ i \neq j.
\end{cases}\]
\end{lemma}

\begin{proof}
First suppose that $\Upsilon'(R)$ is connected. Let $a, b$ be two arbitrary vertices of $\Gamma'(R)$. We may now suppose that $a \in X_i$ and $b \in X_j$. If $i = j$, then $a \nsim b$ in $\Gamma'(R)$. Since $\Upsilon'(R)$ is connected, we have $x_t \in X_t$ such that $x_i \sim x_t$ in $\Gamma'(R)$. Consequently, $a \sim x_t \sim b$ in $\Gamma'(R)$ and $d_{\Gamma'(R)}(a,b) = 2$. If $a \sim b$, then there is nothing to prove. Let $a \nsim b$ in $\Gamma'(R)$. Connectedness of $\Upsilon'(R)$ implies that there exists a path $x_i \sim x_{i_1} \sim x_{i_2} \sim \cdots \sim x_{i_t} \sim x_j$, where $i \neq j$. It follows that $a \sim x_{i_1} \sim x_{i_2} \sim \cdots \sim x_{i_t} \sim b$ in $\Gamma'(R)$ and $d_{\Gamma'(R)}(a,b) = d_{\Upsilon'(R)}(x_i, x_j)$. Therefore, $\Gamma'(R)$ is connected. The converse is straightforward.
\end{proof}

In view of Corollary \ref{corllary_wiener_adjacency}, we have the following proposition.

\begin{proposition}\label{joinof_wiener}
 Let $\Gamma_i'$ be the subgraph induced by the set $X_i$ in $\Gamma'(R)$. Then $\Gamma'(R) = \Upsilon'[\Gamma_1', \Gamma_2', \ldots, \Gamma_k']$.   
\end{proposition}

Let $R$ be a finite commutative ring with unity. As a consequence of Lemma \ref{connectednessof_Upsilon} and  Proposition \ref{joinof_wiener}, we have the following theorem.

\begin{theorem}\label{arbitarywiener}
The Wiener index of the cozero-divisor graph $\Gamma'(R)$ of a finite commutative ring with unity is given by
\[
W(\Gamma'(R)) = 2 \sum {|X_i| \choose 2} + \sum_{\substack{i \neq j\\ 1 \leq i < j \leq k}} |X_i||X_j|d_{\Upsilon'(R)}(x_i, x_j),
\]
where, $x_i$ is a representative of the equivalence class $X_i$ under the relation $\equiv$.
\end{theorem}

In the subsequent sections, we use Theorem \ref{arbitarywiener} to derive the Wiener index of the cozero-divisor graph $\Gamma'(R)$ of various class of rings.

\section{Wiener Index of the cozero-divisor graph of the product of ring of integers modulo $n$}

In this section, we obtain the Wiener index of the cozero-divisor graph $\Gamma'(R)$, when $R \cong \mathbb{Z}_{n_{1}} \times \mathbb{Z}_{n_{2}} \times \cdots \times \mathbb{Z}_{n_{k}}$ or $R \cong \mathbb{Z}_{p_1^{m_1}} \times \mathbb{Z}_{p_2^{m_2}} \times \cdots \times \mathbb{Z}_{p_k^{m_k}}$. For a positive integer $n$, let $d_1, d_2, \ldots, d_t$ be the proper divisors of $n$. Define $\mathcal{A}_{d_i} = \{ x \in V(\Gamma'(\mathbb{Z}_{n}) : \text{gcd}(x,n) =d_i \}$. Moreover, $|\mathcal{A}_{d_{i}}| = \phi(\frac{n}{d_i})$ (cf. \cite{young2015adjacency}) where $\phi$ is the Euler-totient function. Observe that $\mathcal{A}_{d_i}$'s are the equivalence classes of the relation $\equiv$ for the ring $\mathbb{Z}_n$. Now for each $\mathbb{Z}_{p_i^{m_i}}$, let $X_i^{0}, X_i^{1}, \ldots, X_i^{m_i}$ be the corresponding equivalence classes, where $X_{i}^{0} = \{0 \}$, $X_{i}^{1} = U(\mathbb{Z}_{p_i^{m_i}})$ and  $X_{i}^{j} = \mathcal{A}_{p^{j-1}}$ for $2\leq j \leq m_i$. Now we have
\[ |X_i^j|=
 \begin{cases}
 1  & \;\; \; \text{if } j = 0 ,\\
p_i^{m_i} - p_i^{m_i -1} & \;\;\; \text{if } j = 1,\\
p_i^{m_i- j +1} - p_i^{m_i -j}  & \;\; \; \text{if } 2\leq j \leq m_i.
\end{cases}
\]

Let $x = (x_1,x_2, \ldots, x_r, \ldots, x_k)$ and $y = (y_1, y_2, \ldots, y_r, \ldots, y_k) \in R$. Notice that $(x) = (y)$ if and only if  $(x_i) = (y_i)$ for each $i$. It follows that the equivalence classes of the ring $R$ is of the form $X_1^{j_1} \times X_2^{j_2} \times \cdots \times X_k^{j_k}$. Consequently, $|X_1^{j_1} \times X_2^{j_2} \times \cdots \times X_k^{j_k}| = \prod_{i=1}^{k} X_i^{j_i}$.

\begin{lemma}\label{distanceinproductofZ_n}
Let $R \cong \mathbb{Z}_{n_{1}} \times \mathbb{Z}_{n_{2}} \times \cdots \times \mathbb{Z}_{n_{k}}$ and let $x = (x_1,x_2, \ldots, x_r, \ldots, x_k)$, $y = (y_1, y_2, \ldots, y_r, \ldots, y_k) \in R$. Define $S_r = \{(x,y) : ~~ x_r, y_r \in Z(\mathbb{Z}_{n_r})^{*} ~ \text{and }~ (x_r) \subseteq (y_r), x_i = 0, y_i \in U(\mathbb{Z}_{n_i})~ \text{for each } ~i \neq r\}$. Then
    
  \[ d_{\Gamma'(R)}(x,y) =
  \begin{cases}
 1  & \;\; \; \textnormal{if } x \sim y, \\
2 & \;\;\; \textnormal{if } x \nsim y ~~ \text{and }~~ (x,y) \notin S_r,\\
3  & \;\; \; \textnormal{if } (x,y) \in S_r.
\end{cases}
\]
\end{lemma}
\begin{proof}
To prove the result, we discuss the following cases.
    
\noindent\textbf{Case-1.} $x_i \in Z(\mathbb{Z}_{n_i})^{*}$ for each $i \in [k]$. If $x \sim y$ in $\Gamma'(R)$, then $d(x,y) = 1$. Otherwise, either $(x) \subseteq (y)$ or $(y) \subseteq (x)$. Suppose that $y_i \in Z(\mathbb{Z}_{n_i})^{*}$, for each $i \in [k]$. Then for $z = (1, 0, \ldots, 0) \in R$, we obtain $x \sim z \sim y$ in $\Gamma'(R)$. It follows that $d(x, y) = 2$. Now assume that $y_i \in Z(\mathbb{Z}_{n_i})$ for each $i \in [k]$ and $y_j = 0$ for some $j \in [k]$. If $x \nsim y$ in $\Gamma'(R)$, then $(y_i) \subseteq (x_i)$ for each $i$. Choose $z = (z_1, z_2, \ldots, z_k) \in R$ such that $z_i = 0$ whenever $y_i \in Z(\mathbb{Z}_{n_i})^{*}$ and $z_j \in U(\mathbb{Z}_{n_j})$ whenever $y_j = 0$, for some $i, j \in [k]$. Consequently, $x \sim z \sim y$ in $\Gamma'(R)$. It follows that $d(x, y) = 2$. If $y_i \in U(\mathbb{Z}_{n_i})$ and $y_j = 0$, for some $i, j \in [k]$, then note that $d(x, y) = 1$. Now, let $y_i \in U(\mathbb{Z}_{n_i})$ and $y_j \in Z(\mathbb{Z}_{n_i})^{*}$, for some $i, j \in [k]$. If $x \nsim y$ in $\Gamma'(R)$, then $(x_i) \subsetneq (y_i)$ for each $i \in [k]$. Choose $z = (z_1, z_2, \ldots, z_k) \in R$ such that $z_i = 0$ whenever $y_i \in U(\mathbb{Z}_{n_i})^{*}$, and $z_j \in U(\mathbb{Z}_{n_j})$ whenever $y_j \in Z(\mathbb{Z}_{n_j})^{*}$. It follows that $x \sim z \sim y$ in $\Gamma'(R)$ and so $d(x, y) = 2$. Further, assume that $y_i \in U(\mathbb{Z}_{n_i})$ and $y_j \in Z(\mathbb{Z}_{n_j})$ for some $i, j \in [k]$. Then $x \sim y$ in $\Gamma'(R)$ and so $d(x,y) = 1$. 

\noindent\textbf{Case-2.} $x_i \in U(\mathbb{Z}_{n_i})$ and $x_j =0$ for some  $i,j \in [k]$. Suppose $y_{i} \in U(\mathbb{Z}_{n_i})$ and $y_j = 0$ for some $i,j \in [k]$. If $x \sim y$ in $\Gamma'(R)$, then $d(x, y) = 1$. Otherwise, choose $z = (z_1, z_2, \ldots, z_k) \in R$ such that

\[ z_i =
  \begin{cases}
1  & \;\;\;\text{when both} ~~x_i = y_i = 0,\\
0 & \;\;\;\text{otherwise}.
\end{cases}\]

It follows that $d(x,y)=2$. Further, suppose that $y_i \in Z(\mathbb{Z}_{n_i})$ for each $i \in [k]$ and $y_j = 0$ for some $j \in [k]$. If $x \nsim y$ in $\Gamma'(R)$, 
then choose $z=(z_1, z_2, \ldots, z_k)$ such that $z_i \in U(\mathbb{Z}_{n_i})$ whenever $y_i = 0$, and $z_j = 0$ whenever $y_j \in  Z(\mathbb{Z}_{n_j})^{*}$. Consequently, $d(x,y) = 2$. Suppose that $y_i \in Z(\mathbb{Z}_{n_i})^{*}$ and $y_j \in  U(\mathbb{Z}_{n_j})$ for some $i,j \in [k]$. If $x \sim y$ in $\Gamma'(R)$, then $d(x,y) =1$. Otherwise, consider $z = (z_1, z_2, \ldots, z_k) \in R$ such that 

\[ z_i =
\begin{cases}
0  & ~~ \textnormal{if } y_i \in U(\mathbb{Z}_{n_i}), \\
1 & ~~ \textnormal{if } y_i \in Z(\mathbb{Z}_{n_i})^{*}.
\end{cases}
\]
Note that $x \sim z \sim y$ in $\Gamma'(R)$. It follows that $d(x, y) = 2$. Assume that $y_i \in U(\mathbb{Z}_{n_i})$ and $y_j \in Z(\mathbb{Z}_{n_j})$ for some $i,j \in [k]$. If $x \nsim y$ in $\Gamma'(R)$, then choose $z = (z_1, z_2, \ldots, z_k) \in R$ such that $z_i \in U(\mathbb{Z}_{n_i})$ whenever $x_i = 0$, and $z_j = 0$ whenever $x_j \in U(\mathbb{Z}_{n_j})$, for some $i, j \in [k]$. Consequently, $d(x,y) = 2$.

\noindent\textbf{Case-3.} $x_i \in Z(\mathbb{Z}_{n_i})$ for each $i \in [k]$ and $x_j = 0$ for some $j \in [k]$. Suppose $y_i \in Z(\mathbb{Z}_{n_i})$ for each $i \in [k]$ and $y_j = 0$ for some $j \in [k]$. If $x \sim y$ in $\Gamma'(R)$, then $d(x, y) =1$. Let $x \nsim y$ in $\Gamma'(R)$. Then choose $z = (z_1, z_2, \ldots, z_k) \in R$ such that $z_i = 0$, whenever $x_i \in Z(\mathbb{Z}_{n_i})^{*}$, and $z_j = 1$, whenever $x_j = 0$ for some $i, j \in [k]$. It follows that $x \sim z \sim y$ and so $d(x, y) =2$. Next, assume that $y_i \in Z(\mathbb{Z}_{n_i})$ and $y_j \in U(\mathbb{Z}_{n_j})$ for some $i, j \in [k]$. If $x \nsim y$ in $\Gamma'(R)$, then choose $z = (z_1, z_2, \ldots, z_k)$ such that $z_i = 1$ when $x_i = 0$, and $z_j = 0$ when $x_j \in Z(\mathbb{Z}_{n_j})^{*}$ for some $i, j \in [k]$. Consequently, we have $x \sim z \sim y$ in $\Gamma'(R)$. It implies that $d(x,y) =2$. Further, assume that $y_i \in U(\mathbb{Z}_{n_i})$ and $y_j \in Z(\mathbb{Z}_{n_j})^{*}$ for some $i, j \in [k]$. Let $x \nsim y$ in $\Gamma'(R)$. Suppose that there exists $r \in [k]$ such that $x_r \in Z(\mathbb{Z}_{n_r})^{*}$ and $x_i = 0$ for each $i \in [k]\setminus \{r\}$. Also, $y_i \in U(\mathbb{Z}_{n_i})$ and $y_r \in Z(\mathbb{Z}_{n_r})^{*}$ for each $i \in [k]\setminus \{r\}$. Then $(x_r) \subsetneq (y_r)$. If there exists $a = (a_1, a_2, \ldots,a_r, \ldots, a_k)$ such that $a \sim y$, then $(y_r) \subsetneq (a_r)$. It follows that $(x_r) \subsetneq (y_r) \subsetneq (a_r)$. Consequently, $a \nsim x$ in $\Gamma'(R)$. Therefore,  $d(x,y) > 2$. Consider $z= (z_1, z_2, \ldots, z_k)$ and $z'= (z_1', z_2', \ldots, z_k') \in R$ such that 
\[ z_i =
\begin{cases}
 1  &  ~~  \textnormal{if } x_i = 0, \\
0 & ~~ \textnormal{if } x_i \in Z(\mathbb{Z}_{n_i})^{*}
\end{cases}
\]
and
\[ z_i' =
\begin{cases}
0  & ~~ \textnormal{if } y_i \in U(\mathbb{Z}_{n_i}), \\
1 & ~~ \textnormal{if } y_i \in Z(\mathbb{Z}_{n_i})^{*}.
\end{cases}
\]
It follows that $x \sim z \sim z' \sim y$ in $\Gamma'(R)$. Therefore, $d(x,y) =3$. Next, we claim that if there exist $t$ and $r \in [k]$ such that $x_t \in Z(\mathbb{Z}_{n_t})^{*}, x_r \in Z(\mathbb{Z}_{n_r})^{*}$ then $d(x,y) \leq 2$. If $x \sim y$ in $\Gamma'(R)$, then $d(x,y) = 1$. Next, assume that $x \nsim y$ in $\Gamma'(R)$. Since $x \nsim y$, we have $(x) \subsetneq (y)$. If there exists $i_1 \in [k]$ such that $x_{i_1}, y_{i_1} \in  Z(\mathbb{Z}_{n_{i_1}})^{*}$ then take $r=i_1$. Now consider $z = (z_1,z_2, \ldots,z_k) \in R$ such that $z_t = 0, z_r = 1$ and, for $i \neq \{t, r\}$ whenever $y_i \in U(\mathbb{Z}_{n_i})$ take $z_i =0$ and, whenever $y_i \in Z(\mathbb{Z}_{n_i})^{*}$ then choose $z_i =1$. It follows that $x \sim z \sim y$ in $\Gamma'(R)$. Therefore, $d(x,y) \leq 2$.

\noindent\textbf{Case-4.} $x_i \in Z(\mathbb{Z}_{n_i})^{*}$ and $x_j \in U(\mathbb{Z}_{n_j})$ for some $i, j \in [k]$. Let $y_i \in Z(\mathbb{Z}_{n_i})^{*}$ and $y_j \in U(\mathbb{Z}_{n_j})$ for some $i, j \in [k]$. If $x \sim y$ in $\Gamma'(R)$, then $d(x,y) =1$. Let $x \nsim y$ in $\Gamma'(R)$. Then choose $z = (z_1, z_2, \ldots, z_k) \in R$ such that $z_i = 0$ whenever $x_i \in U(\mathbb{Z}_{n_i})$, and $z_j = 1$ whenever $x_j  \in Z(\mathbb{Z}_{n_j})^{*}$ for some $i, j \in [k]$. It follows that $d(x,y) = 2$. Next, let $y_i \in Z(\mathbb{Z}_{n_i})$ and $y_j \in U(\mathbb{Z}_{n_j})$ for some $i, j \in [k]$. If $x \nsim y$ in $\Gamma'(R)$, then choose $z= (z_1, z_2, \ldots, z_k)$ such that $z_i = 1$ whenever $x_i \in Z(\mathbb{Z}_{n_i})^{*}$, and $z_j = 0$ whenever $x_j \in U(\mathbb{Z}_{n_j})$ for some $i, j \in [k]$. Therefore, $d(x,y) =2$.

\noindent\textbf{Case-5.} $x_i \in Z(\mathbb{Z}_{n_i})$ and $x_j \in U(\mathbb{Z}_{n_j})$ for some $i, j \in [k]$. Assume that $y_i \in Z(\mathbb{Z}_{n_i})$ and $y_j \in U(\mathbb{Z}_{n_j})$ for some $i, j \in [k]$. If $x \sim y$ in $\Gamma'(R)$, then $d(x,y) =1$. Otherwise, choose $z= (z_1, z_2, \ldots, z_k) \in R$ as follows 

\[ z_i =
\begin{cases}
 0 & ~~ \textnormal{if } x_i \in Z(\mathbb{Z}_{n_i})^{*} ~~\text{and } x_i \in U(\mathbb{Z}_{n_i}), \\
1 & ~~ \textnormal{if } x_i = 0.
\end{cases}
\]
Then $x \sim z \sim y$ in $\Gamma'(R)$. It follows that $d(x,y) = 2$.
\end{proof}

In view of Lemma \ref{distanceinproductofZ_n}, now we calculate the Wiener index of $\Gamma'(R)$. Let $x = (x_1^{j_1}, x_2^{j_2}, \ldots, x_k^{j_k})$ and $y = (y_1^{l_1}, y_2^{l_2}, \ldots, y_k^{l_k})$ be the representatives of two distinct equivalence classes $X_1^{j_1} \times X_2^{j_2} \times \cdots \times X_k^{j_k}$ and $X_1^{l_1} \times X_2^{l_2} \times \cdots \times X_k^{l_k}$, respectively. 

\begin{theorem}
  The Wiener index of the cozero-divisor graph $\Gamma'(R)$, where $R \cong \mathbb{Z}_{p_1^{m_1}} \times \mathbb{Z}_{p_2^{m_2}} \times \cdots \times \mathbb{Z}_{p_k^{m_k}}$, is given below:
  
  \begin{align*}
    W(\Gamma'(R)) &= 2 \sum_{(x_1^{j_1}, x_2^{j_2}, \ldots x_k^{j_k}) \in \Upsilon'} { { \prod_{\substack{i=1 \\ j_i\ge 1}}^{k} (p_i^{m_i - j_i +1} - p_i^{m_i - j_i} ) } \choose 2 } + \sum_{x \sim y} \left( \prod_{\substack{i=1 \\ j_i \ge 1}}^{k} (p_i^{m_i - j_i +1} - p_i^{m_i - j_i } ) \right) \left( \prod_{\substack{i=1 \\ j_i \ge 1}}^{k}(p_i^{m_i - l_i +1} - p_i^{m_i - l_i } ) \right)  \\
     &+ 2 \sum_{\substack{x \nsim y \\ (x,y) \notin S_r}} \left( \prod_{\substack{i=1 \\ j_i \ge 1}}^{k} (p_i^{m_i - j_i +1} - p_i^{m_i - j_i } ) \right) \left( \prod_{\substack{i=1 \\ j_i \ge 1}}^{k}(p_i^{m_i - l_i +1} - p_i^{m_i - l_i } ) \right) \\
     &+ 3 \sum_{{(x,y) \in S_r}} \left( \prod_{\substack{i=1 \\ j_i \ge 1}}^{k} (p_i^{m_i - j_i +1} - p_i^{n_i - j_i} ) \right) \left( \prod_{\substack{i=1 \\ j_i \ge 1}}^{k}(p_i^{m_i - l_i +1} - p_i^{m_i - l_i} ) \right).
  \end{align*}

\end{theorem}

\begin{example}
Let $R \cong \mathbb{Z}_2 \times \mathbb{Z}_4 \times \mathbb{Z}_9$. Then $|X_1^0| = 1$, $|X_2^0| = 1$, $|X_3^0| = 1$, $|X_1^1| = 1$, $|X_2^1| = 2$, $|X_3^1| = 6$, $|X_1^2| = 0$, $|X_2^2| = 1$ and $|X_3^2| = 2$. Let $Y_1 = X_1^0 \times X_2^0 \times X_3^1$, $Y_2 = X_1^0 \times X_2^0 \times X_3^2$, $Y_3 = X_1^0 \times X_2^1 \times X_3^0$, $Y_4 = X_1^0 \times X_2^2 \times X_3^0$, $Y_5 = X_1^0 \times X_2^1 \times X_3^1$, $Y_6 = X_1^0 \times X_2^2 \times X_3^1$, $Y_7 = X_1^0 \times X_2^1 \times X_3^2$, $Y_8 = X_1^0 \times X_2^2 \times X_3^2$, $Y_9 = X_1^1 \times X_2^0 \times X_3^0$, $Y_{10} = X_1^1 \times X_2^0 \times X_3^1$, $Y_{11} = X_1^1 \times X_2^0 \times X_3^2$, $Y_{12} = X_1^1 \times X_2^1 \times X_3^0$, $Y_{13} = X_1^1 \times X_2^1 \times X_3^2$, $Y_{14} = X_1^1 \times X_2^2 \times X_3^0$, $Y_{15} = X_1^1 \times X_2^2 \times X_3^1$ and $Y_{16} = X_1^1 \times X_2^2 \times X_3^2$. Then $S_3 = \{ \{ Y_2, Y_{13}\} \}, \ S_2 = \{ \{Y_4, Y_{15}\} \}$ and the pair of sets whose elements are at distance two

$ \{ \{ Y_1, Y_2\}, \{ Y_1, Y_5\}, \ \{ Y_1, Y_6\}, \ \{ Y_1, Y_{10}\}, \ \{ Y_1, Y_{15}\}, \ \{ Y_2, Y_5\}, \ \{ Y_2, Y_6\}, \ \{ Y_2, Y_7\}, \ \{ Y_2, Y_8\}, \ \{ Y_2, Y_{10}\}, \ \{ Y_2, Y_{11}\}, \ \{ Y_2, Y_{15}\}, \\ \{ Y_2, Y_{16}\}, \ \{ Y_3, Y_{4}\}, \ \{ Y_3, Y_{5}\}, \ \{ Y_3, Y_{7}\}, \ \{ Y_3, Y_{12}\}, \ \{ Y_3, Y_{13}\}, \ \{ Y_4, Y_{5}\}, \ \{ Y_4, Y_{6}\}, \ \{ Y_4, Y_{7}\}, \ \{ Y_4, Y_{8}\}, \ \{ Y_4, Y_{12}\}, \ \{ Y_4, Y_{13}\}, \\ \{ Y_4, Y_{14}\}, \ \{ Y_4, Y_{16}\}, \ \{ Y_5, Y_{6}\}, \ \{ Y_5, Y_{7}\}, \ \{ Y_5, Y_{8}\}, \ \{ Y_6, Y_{8}\},  \ \{ Y_6, Y_{15}\}, \ \{ Y_7, Y_{8}\}, \ \{ Y_7, Y_{13}\},  \ \{ Y_8, Y_{13}\}, \ \{ Y_8, Y_{15}\}, \ \{ Y_8, Y_{16}\}, \\ \{ Y_9, Y_{10}\}, \ \{ Y_9, Y_{11}\}, \ \{ Y_9, Y_{12}\}, \ \{ Y_9, Y_{13}\}, \ \{ Y_9, Y_{14}\}, \ \{ Y_9, Y_{15}\}, \ \{ Y_9, Y_{16}\}, \ \{ Y_{10}, Y_{11}\}, \ \{ Y_{10}, Y_{15}\}, \ \{ Y_{11}, Y_{13}\}, \ \{ Y_{11}, Y_{15}\}, \\ \{ Y_{11}, Y_{16}\}, \ \{ Y_{12}, Y_{13}\}, \ \{ Y_{12}, Y_{14}\}, \ \{ Y_{13}, Y_{14}\}, \ \{ Y_{13}, Y_{16}\}, \ \{ Y_{14}, Y_{15}\}, \ \{ Y_{14}, Y_{16}\}, \ \{ Y_{15}, Y_{16}\} \} $.

Thus, the Wiener index of the cozero-divisor graph of the ring $\mathbb{Z}_2 \times \mathbb{Z}_4 \times \mathbb{Z}_9$ is
\begin{align*}
     W(\Gamma'(\mathbb{Z}_2 \times \mathbb{Z}_4 \times \mathbb{Z}_9)) &= 2 \times \dfrac{1}{2} \left[ 30 + 2 + 2 +0 +132+30+12+2+0+30+2+2+12+0+30+2 \right] \\
  &+ [ 6(2+1+4+2+1+2+2+4+1+2) + 2( 2+1+1+2+1)\\
  &+ 2(6+2+1+6+2+1+6+2) +(1+6+2)+ 12(1+6+2+2+4+1+6+2) \\
  &+ 6(4+1+6+2+2+4+1+2) + 4(1+6+2+2+1+6+2) + 2(1+6+2+2+1) + (0) \\
  &+ 6(2+4+1+2) + 2(2+1) + 2(6+2) + 4(6) + (0) ] \\
  &+ 2[ 6(2+12+6+6+6) + 2 (12+6+4+2+6+2+6+2) +2(1+12+4+2+4) \\
  &+ (12+6+4+2+2+4+1+2) +12(6+4+2) + 6(2+6) + 4(2+4) + 2(4+6+2) \\
  &+ (6+2+2+4+1+6+2) + 6(2+6) +2(4+6+2) +2(4+1) +4 (1+2) + (6+2) + 6(2)] \\
  &+ 3 [ (2 \times 4) +  (1 \times 6) ] \\
  &= 2611
\end{align*}
\end{example}


\section{The Wiener Index of the cozero-divisor graph of reduced ring}

In this section, we obtain the Wiener index of the cozero-divisor graph of a finite commutative reduced ring. Let $R$ be a reduced ring i.e. $R \cong F_{q_1} \times F_{q_2} \times \cdots \times F_{q_k}$ with $k \geq 2$, where $F_q$ is a finite field with $q$ elements. Notice that, for $x = (x_1 , x_2, \ldots, x_k)$ and $y = (y_1 , y_2, \ldots, y_k) \in R$ such that $(x) = (y)$, we have $x_i = 0 $ if and only if $y_i = 0$ for each $i$. For $i_1, i_2, \ldots, i_r \in [k]$, define 
\[
X_{\{i_1,i_2, \ldots, i_r\}} = \{ (x_1, x_2, \ldots, x_k) \in R : \text{only} \ x_{i_1}, x_{i_2}, \ldots, x_{i_r} \text{ are non-zero} \}. 
\]

 Note that the sets $X_A$, where $A$ is a non-empty proper subset of $[k]$, are the equivalence classes under the relation $\equiv$. We write $x_A$ by the representative of equivalence class $X_A$. Now we obtain the possible distances between the vertices of $\Upsilon'(R)$.

\begin{lemma}\label{reducedringdistances}
    For the distinct vertices $x_A$ and $x_B$ of $\Upsilon'(R)$ we have
 \[   
 d_{\Upsilon'(R)}(x_A, x_B) = 
 \begin{cases}
  1 & \;\; \text{if} \ A \nsubseteq B \ \text{and} \ B \nsubseteq A \\
  
 2 & \;\; \text{otherwise}.
 \end{cases}
    \]
\end{lemma}

\begin{proof}
First assume that  $A \nsubseteq B \ \text{and} \ B \nsubseteq A$. Then $(x_A) \nsubseteq (x_B) $ and $(x_B) \nsubseteq (x_A)$. It follows that $d_{\Upsilon'}(x_A, x_B)=1 $. Now without loss of generality let $ A \subsetneq B$. Then there exists $i \in [k]$ such that $i \notin B$ and so $i \notin A$. Then by Lemma \ref{adjacencyofclasses}, we have $x_A \sim x_{\{i\}} \sim x_B$. Thus, $d_{\Upsilon'}(x_A, x_B)=2 $.
\end{proof}

For distinct $A,B \subsetneq [k]$, we define $ D_1 = \{  \{A, B\} :  \  A \nsubseteq B \} $ and $ D_2 = \{  \{A, B\} : \  A \subsetneq B  \} $. Using Theorem \ref{arbitarywiener} and the sets $D_1$ and $D_2$, we obtain the Wiener index of the cozero-divisor $\Gamma'(R)$ of a reduced ring $R$ in the following theorem.

\begin{theorem}\label{reducedringwiener}
The Wiener index of the cozero-divisor graph of a finite commutative reduced ring $R \cong F_{q_1} \times F_{q_2} \times \cdots \times F_{q_k}$, $k \geq 2$, is given by
  \begin{align*}
      W(\Gamma'(R)) &= 2 \sum_{A \subset [k]} {\prod_{i \in A}{(q_i -1)} \choose 2} + \sum_{\{A,B\} \in D_1} {\left( \prod_{i \in A}(q_i -1) \right) \left( \prod_{j \in B}(q_j -1) \right)} \\
      &+ 2 \sum_{\{A,B\} \in D_2} {\left( \prod_{i \in A}(q_i -1) \right) \left( \prod_{j \in B}(q_j -1) \right).}
 \end{align*}
\end{theorem}

\begin{proof}
The proof follows from Lemma \ref{reducedringdistances}.
\end{proof}

\begin{example}{{\cite[Corollary 6.2]{a.mathil2022cozero}}}
Let $R = \mathbb{Z}_{pq} \cong \mathbb{Z}_p \times \mathbb{Z}_q $, where $p,q$ are distinct prime numbers. Then we have two distinct equivalence classes, $X_{\{1\}} = \{ (a,0) : a \in \mathbb{Z}_p \setminus \{0\} \}$ and $X_{\{2\}} = \{ (0,b) : b \in \mathbb{Z}_q \setminus \{0\} \}$, of the equivalence relation $\equiv$. Moreover, $D_1 = \{ \{\{1\}, \{2\}\}\}$  and $D_2 = \{ \ \} $. Note that $|X_{\{1\}}| = p-1$ and $|X_{\{2\}}| = q-1$. Consequently, by Theorem \ref{reducedringwiener}, we get $W(\Gamma'(\mathbb{Z}_{pq})) = (p-1)(p-2) + (q-1)(q-2) + (p-1)(q-1) = p^2 + q^2 -4p -4q + pq +5$.

\end{example}

\begin{example}
Let $R = \mathbb{Z}_{pqr} \cong \mathbb{Z}_p \times \mathbb{Z}_q \times \mathbb{Z}_r$, where $p,q,r$ are distinct prime numbers. For $a \in \mathbb{Z}_p \setminus \{0\}$, $b \in \mathbb{Z}_q \setminus \{0\} $ and $c \in \mathbb{Z}_q \setminus \{0\}$, we have the equivalence classes : $X_{\{1\}} = \{ (a,0,0) \}, \ X_{\{2\}} = \{ (0,b,0) \}, \ X_{\{3\}} = \{ (0,0,c) \}, \ X_{\{1,2\}} = \{ (a,b,0) \}, \ X_{\{1,3\}} = \{ (a,0,c) \}, \ X_{\{2,3\}} = \{ (0,b,c) \} $. Moreover, $D_1 = \bigl\{ \{ \{ 1\}, \{ 2\}\}, \ \{ \{ 1\}, \{ 3\}\}, \ \{ \{ 2\}, \{ 3\}\}, \ \{ \{1, 2\}, \{ 1,3\}\},\\ 
\{ \{ 1,2\}, \{ 2,3\}\}, \ \{ \{1,3\}, \{ 2,3\}\}, \ \{ \{ 1\}, \{ 2,3\}\}, \ \{ \{ 2\}, \{ 1,3\}\}, \ \{ \{ 3 \}, \{ 1,2\}\} \bigr\}$ and $D_2 = \bigl\{ \{ \{ 1\}, \{ 1,2\}\}, \ \{ \{ 1\}, \{ 1,3\}\}, \\ \{ \{ 2\}, \{ 1,2\}\}, \ \{ \{ 2\}, \{ 2,3\}\},\ 
\{ \{ 3\}, \{ 1,3\}\}, \ \{ \{ 3\}, \{ 2,3\}\} \bigr\}$. Also, $|X_{\{1\}}| = (p-1)$, $|X_{\{2\}}| = (q-1)$, $|X_{\{3\}}| = (r-1)$, $|X_{\{1,2\}}| = (p-1)(q-1)$, $|X_{\{1,3\}}| = (p-1)(r-1)$, $|X_{\{2,3\}}| = (q-1)(r-1)$. Then, by Theorem \ref{reducedringwiener}, the Wiener index of $\Gamma'(R)$ is given by

\begin{align*}
    W(\Gamma'(\mathbb{Z}_{pqr})) &= 2{p-1 \choose 2} +  2{q-1 \choose 2} + 2{r-1 \choose 2} + 2{(p-1)(q-1) \choose 2} + 2{(p-1)(r-1) \choose 2} + 2{(q-1)(r-1) \choose 2} \\
    &+ (p-1)(q-1) + (p-1)(r-1) + (q-1)(r-1) + (p-1)(q-1)(p-1)(r-1) + (p-1)(q-1)(q-1)(r-1) \\
    &+ (p-1)(r-1)(q-1)(r-1) + (p-1)(q-1)(r-1) + (q-1)(p-1)(r-1) + (r-1)(p-1)(q-1) \\
    &+ 2(p-1)\left[ (p-1)(q-1) \right]+ 2(p-1)\left[ (p-1)(r-1) \right]+ 2(q-1)\left[ (p-1)(q-1) \right] + 2(q-1)\left[ (q-1)(r-1) \right] \\
    &+ 2(r-1)\left[ (p-1)(r-1) \right]+ 2(r-1)\left[ (q-1)(r-1) \right]. 
\end{align*}

simplifying this expression, we get
\[
W(\Gamma'(\mathbb{Z}_{pqr})) = pqr(p + q + r -3) + p^2q^2 + p^2r^2 + q^2r^2 -p^2(q+r) - q^2(p+r) -r^2(p+q) - 2(pq +pr +qr) +4(p+q+r) - 3.
\]
\end{example}

\textbf{Remark:} For the ring $\mathbb{Z}_{n}$ of integers modulo $n$, the equivalence classes with respect to the relation $\equiv$ are the sets $\mathcal{A}_{d_{i}}$, where $d_i$'s are the proper divisors of $n$ and $\mathcal{A}_{d_{i}} = \{ x \in \mathbb{Z}_{n}: \text{gcd}(x,n) = d_{i}\}$. It is known that $|\mathcal{A}_{d_{i}}| = \phi(\frac{n}{d_i})$ for $1 \leq i \leq k$, where $\phi$ is Euler totient function (see \cite{young2015adjacency}). Thus, we have the following corollaries of Theorem \ref{arbitarywiener}.


\begin{corollary}{{\cite[Theorem 6.1]{a.mathil2022cozero}}}
 For $n = p_1p_2 \cdots p_k$, where $p_i$'s are distinct primes and $2 \le k \in \mathbb{N}$, we have \[W(\Gamma'(\mathbb{Z}_n)) = \sum_{i=1}^{2^k-2} \phi(\dfrac{n}{d_i})\left(\phi(\dfrac{n}{d_i})-1\right) \ + \  \dfrac{1}{2}\sum_{\substack{d_i \nmid d_j \\ d_j \nmid d_i}} \phi(\dfrac{n}{d_i})\phi(\dfrac{n}{d_j}) \ + \ 2\sum_{\substack{d_i \mid d_j \\ i\neq j}} \phi(\dfrac{n}{d_i})\phi(\dfrac{n}{d_j}),\] 
 where $d_i$'s are the proper divisors of $n$.
\end{corollary}

Let $\tau(n)$ be the number of divisors of $n$ and let $D= \{ d_1, d_2, \ldots ,d_{\tau(n)-2} \}$ be the set of all proper divisors of $n = p_1^{n_1}p_2^{n_2}\cdots p_r^{n_r} \cdots p_k^{n_k}$ with $k \geq 2$. If $d_i \mid d_j$, then define
\begin{align*}
A &= \{ (d_i, d_j) \in D \times D \  | \ d_i \neq p_r^s \}; \\
B &= \{ (d_i, d_j) \in D \times D \  | \ d_i = p_r^s \ \text{and} \ \dfrac{n}{d_j} \neq p_r^t \}; \\
C &= \{ (d_i, d_j) \in D \times D \  | \ d_i = p_r^s \ \text{and} \  \dfrac{n}{d_j} = p_r^t \}.
\end{align*}

\begin{corollary}{{\cite[Theorem 6.3]{a.mathil2022cozero}}} With the above defined notations, for $n = p_1^{n_1}p_2^{n_2}\cdots p_r^{n_r} \cdots p_k^{n_k}$ with $k \geq 2$ and $p_i$'s are distinct primes, we have
 \[W(\Gamma'(\mathbb{Z}_n)) = \sum_{i=1}^{\tau (n)-2} \phi(\dfrac{n}{d_i}) \left(\phi(\dfrac{n}{d_i}) -1 \right) \ + \ \dfrac{1}{2}\sum_{\substack{d_i \nmid d_j \\ d_j \nmid d_i}}\phi(\dfrac{n}{d_i})\phi(\dfrac{n}{d_j}) \ + \ 2\sum_{(d_i, d_j) \in A}\phi(\dfrac{n}{d_i})\phi(\dfrac{n}{d_j}) 
 + \ 2\sum_{(d_i, d_j) \in B}\phi(\dfrac{n}{d_i})\phi(\dfrac{n}{d_j})\]  
 \[ \hspace{-10cm} + ~~~         3\sum_{(d_i, d_j) \in C}\phi(\dfrac{n}{d_i})\phi(\dfrac{n}{d_j}).\]

\end{corollary}


\section{SageMath Code } 

In this section, we produce a SAGE code to compute the Wiener index of the cozero-divisor graph of ring classes considered in this paper including the ring $\mathbb{Z}_n$ of integers modulo $n$. On providing the value of integer $n$, the following SAGE code computes the Wiener index of the graph $\Gamma'(\mathbb{Z}_n)$.

\vspace{.3cm}
\lstset{language=Python}
\begin{lstlisting}
cozero_divisor_graph=Graph()
E=[]
n=72

for i in range(n):
    for j in range(n):
        if(i!=j):
            p=gcd(i,n)
            q=gcd(j,n)
            if (p%q!=0 and q%p!=0):
                E.append((i,j))
           
cozero_divisor_graph.add_edges(E)

if(E==[]): 
    V=[]
    for i in range(1,n):
        if (gcd(i,n)!=1):
            V.append(i)
    cozero_divisor_graph.add_vertices(V)
    
    
W=cozero_divisor_graph.wiener_index();

if (W==oo):
    print("Wiener Index undefined for Null Graph")
else :
    print("Wiener Index:", W)

\end{lstlisting}

Using the given code, in the Table \ref{Z_nwiener}, we obtain the Wiener index of $\Gamma'(\mathbb{Z}_n)$ for some values of $n$.

\vspace{.5cm}
 \begin{center}
\begin{tabular}{|c| c| c| c| c | c | c |}
\hline
 $n$ & $100$ & $500$ & $1000$ & $1500$ & $2000$ & $2500$ \\
 \hline
 $W(\Gamma'(\mathbb{Z}_n))$ & $2954$ & $77174$ & $306202$ & $930248$ & $1222530$ & $1946274$  \\
\hline
\end{tabular}

\vspace{.2cm}
\captionof{table}{Wiener index of $\Gamma'(\mathbb{Z}_n$)}
\label{Z_nwiener}
\end{center}

\vspace{0.5cm}
Let $R$ be a reduced ring i.e. $R \cong F_{q_1} \times F_{q_2} \times \cdots \times F_{q_n} $, where $F_{q_i}$ is a field with $q_i$ elements. The following code determines the Wiener index of $\Gamma'(R)$ on providing the values of the field size $q_i$ ($1 \le i \le n$).

\vspace{.3cm}
\lstset{language=Python}
\begin{lstlisting}
field_orders=[3,5,7]
P=Subsets(range(len(field_orders)))[1:-1]
P=[Set(i) for i in P]

D1=[]
D2=[]
for i in P:
    for j in P:
        if (not(i.issubset(j) or j.issubset(i)) and P.index(i) > P.index(j)):
            D1.append([i,j])
        if (i.issubset(j) and i!=j):
            D2.append([i,j])

partial_sum=0
for i in P:
    sum_pp=1
    for j in i:
        sum_pp *= field_orders[j]-1
    partial_sum +=((sum_pp*(sum_pp-1))/2)
    
D1_sum=0
for i in D1:
    D1_pp=1
    for j in i[0]:
        D1_pp *= field_orders[j]-1
    for k in i[1]:
        D1_pp *= field_orders[k]-1
    D1_sum += D1_pp
    
D2_sum=0
for i in D2:
    D2_pp=1
    for j in i[0]:
        D2_pp *= field_orders[j]-1
    for k in i[1]:
        D2_pp *= field_orders[k]-1
    D2_sum += D2_pp
    
W = 2*partial_sum + D1_sum + 2*D2_sum
print("Wiener Index:" , W)


\end{lstlisting}

Using the given code, in the following tables, we obtain the Wiener index of the cozero-divisor graphs of the reduced rings $F_{q_1} \times F_{q_2}$ (see Table \ref{F1F2wiener}) and $F_{q_1} \times F_{q_2} \times F_{q_3}$ (see Table \ref{F1F2F3wiener}), respectively. 

\vspace{.5cm}
 \begin{center}
\begin{tabular}{|c| c| c| c| c | c | c |}
\hline
 $(q_1, q_2)$ & $(9, 25)$ & $(49, 81)$ & $(101, 121)$ & $(125, 139)$ & $(163, 169)$ & $(289, 343)$ \\
 \hline
 $W(\Gamma'(F_{q_1} \times F_{q_2}))$ & $800$ & $12416$ & $36180$ & $51270$ & $81354$ & $297774$  \\
\hline
\end{tabular}

\vspace{.2cm}
\captionof{table}{Wiener index of $\Gamma'(F_{q_1} \times F_{q_2})$}
\label{F1F2wiener}
\end{center}

\vspace{0.5cm}
\begin{center}
\begin{tabular}{|c| c| c| c| c | c | c |}
\hline
 $(q_1, q_2, q_3)$ & $(7,8, 13)$ & $(9, 25, 49)$ & $(53, 64, 81)$ & $(83, 101, 121)$ & $(125, 131, 169)$ & $(289, 343, 361)$ \\
 \hline
 $W(\Gamma'(F_{q_1} \times F_{q_2} \times F_{q_3}))$ & $35196$ & $2500400$ & $108637254$ & $620456582$ & $2355211790$ & $71251552134$  \\
\hline
\end{tabular} 
\vspace{.2cm}
\captionof{table}{Wiener index of $\Gamma'(F_{q_1} \times F_{q_2} \times F_{q_3})$}
\label{F1F2F3wiener}
\end{center}

\vspace{0.5cm}
Let $R \cong \mathbb{Z}_{p_1^{m_1}} \times \mathbb{Z}_{p_2^{m_2}} \times \cdots \times \mathbb{Z}_{p_k^{m_k}}$. Then the following SAGE code gives the value of $W(\Gamma'(R))$ after providing the values of $p_i^{m_i}$($1 \le i \le k)$, where each $p_i$ is a prime.

\vspace{0.5cm}
\lstset{language=Python}
\begin{lstlisting}
orders = [2,4,9]
A = cartesian_product([range(i) for i in orders]).list()
units = [{i for i in range(1,j) if gcd(i,j) == 1} for j in orders]

def contQ(lst1, lst2):
    flag = True
    for i in range(len(orders)):
        p=gcd(lst1[i],orders[i])
        q=gcd(lst2[i],orders[i])
        if(not(lst1[i]==0 or {lst2[i]}.issubset(units[i]) or p%q==0)):
            flag = False
    return flag

E=[]
for i in A:
    for j in A:
        if(not(contQ(i,j) or contQ(j,i))and A.index(i) > A.index(j)):
            E.append([i,j])
        
G = Graph()
G.add_edges(E)
W=G.wiener_index()
print("Wiener_Index:", W)
\end{lstlisting}

Using the given code, we obtain the Wiener index of the cozero-divisor graph of the ring $R\cong \mathbb{Z}_{p_1^{m_1}} \times \mathbb{Z}_{p_2^{m_2}} \times \cdots \times \mathbb{Z}_{p_k^{m_k}}$ (see Table \ref{Z_nproductwiener}).

\vspace{0.5cm}
\begin{center}
\begin{tabular}{|c| c|}
\hline
 $ R$ & $ W(\Gamma'(R))$ \\
 \hline
 $ \mathbb{Z}_{4} \times \mathbb{Z}_{9}$  & $420$ \\
\hline
 $ \mathbb{Z}_{9} \times \mathbb{Z}_{25} $  & $8808$ \\
\hline
$ \mathbb{Z}_{16} \times \mathbb{Z}_{25}$  & $48870$ \\
\hline
$ \mathbb{Z}_{27} \times \mathbb{Z}_{49}$  & $268022$ \\
\hline
$ \mathbb{Z}_{2} \times \mathbb{Z}_{4} \times \mathbb{Z}_{4}$  & $521$ \\
\hline
$ \mathbb{Z}_{5} \times \mathbb{Z}_{7} \times \mathbb{Z}_{11}$  & $14948$ \\
\hline
$ \mathbb{Z}_{8} \times \mathbb{Z}_{9} \times \mathbb{Z}_{16}$  & $167769$ \\
\hline
$ \mathbb{Z}_{4} \times \mathbb{Z}_{9} \times \mathbb{Z}_{25}$  & $327394$ \\
\hline
$ \mathbb{Z}_{2} \times \mathbb{Z}_{4} \times \mathbb{Z}_{9} \times \mathbb{Z}_{9}$ & $232937$ \\
\hline
$ \mathbb{Z}_{3} \times \mathbb{Z}_{4} \times \mathbb{Z}_{8} \times \mathbb{Z}_{8}$ & $333963$ \\
\hline
\end{tabular}

\vspace{.2cm}
\captionof{table}{Wiener index of $\Gamma'(R)$}
\label{Z_nproductwiener}
\end{center}

\vspace{.3cm}
	
\textbf{Acknowledgments:} The first author gratefully acknowledge for providing financial support to CSIR (09/719(0093)/2019-EMR-I) government of India. 
\vspace{.3cm}

\textbf{Conflict of interest:} On behalf of all authors, the corresponding author states that there is no conflict of interest.


\end{document}